\documentclass[12pt]{article}
\usepackage[pdftex]{hyperref}
\hypersetup{
  colorlinks   = true, 
  urlcolor     = gray, 
  linkcolor    = red, 
  citecolor   =  blue
}
\usepackage{amsmath,amsfonts,amsthm,amssymb,graphicx,xcolor}
\usepackage[all,2cell,ps]{xy}
\hypersetup{colorlinks=true, citecolor=[rgb]{0,0.5,1}}
\usepackage[margin=1.5 in]{geometry}

\theoremstyle{plain}
\newtheorem{theorem}[equation]{Theorem}

\newtheorem{lemma}[equation]{Lemma}

\theoremstyle{definition}
\newtheorem{define}[equation]{Definition}

\newtheorem{remark}[equation]{Remark}

\newtheorem{question}[equation]{Question}

\newcommand{\IR}{\mathbb{R}}

\newcommand{\IZ}{\mathbb{Z}}

\title{Extended Graph $4$-Manifolds, and Einstein Metrics}
\author{\small{Luca F. Di Cerbo}\footnote{Supported in part by NSF grant DMS-2104662} \\ \scriptsize{University of Florida}\\ \footnotesize{\textsf{ldicerbo@ufl.edu}}}
\date{}

\begin{document}

\maketitle


\begin{abstract}
We show that extended graph $4$-manifolds (as defined by Frigerio-Lafont-Sisto in \cite{LafontBook}) do not support Einstein metrics.  
\end{abstract}

\vspace{10 cm}


\tableofcontents

\vspace{1cm}


\section{Introduction}

The study of the interplay between curvature and topology is the aim of a considerable part modern Riemannian geometry. The study of Einstein metrics is currently an active line of research in the field; see for example the authoritative book of Besse \cite{Bes87}, and the classical survey by LeBrun and Wang \cite{LW01}. Recall that a Riemannian manifold $(M, g)$ is said to be Einstein if its Ricci tensor is proportional to the metric:
\[
Ric_g = \lambda g,
\]
where $\lambda\in\IR$ is a constant. In real dimension two and three this condition is equivalent to the constancy of the sectional curvature, so that the study of Einstein metrics in these dimensions reduces to the study of real space-forms. As a result, most $3$-manifolds do not admit Einstein metrics! Indeed, it follows from the works of Thurston, Hamilton, and Perelman (see for example the survey by Lott and Kleiner \cite{LK08}) that any irreducible closed 3-manifolds can be manufactured as a graph of Seifert and real-hyperbolic pieces (the vertices) glued along incompressible two dimensional tori (the edges). In other words, any 3-manifold whose graph-like structure in not a single vertex cannot support an Einstein metric.

In higher dimensions, the idea of manufacturing interesting closed manifolds by imitating the construction of graph-like 3-manifolds has appeared in many different contexts. In particular, the idea of doubling real-hyperbolic manifolds with cusps has been extensively studied, see for example \cite{Hei76}, \cite{Ont02}, and \cite{AF09}. Recently, Frigerio-Lafont-Sisto in \cite{LafontBook} have identified and studied a large class of higher dimensional manifolds which nicely generalizes many of the 3-dimensional constructions. They call these manifolds \emph{extended graph manifolds}, see Definition 0.2 in \cite{LafontBook}. Roughly speaking, they are decomposed into finitely many pieces (the vertices), each vertex is a manifold with boundary tori (the edges) and the various pieces are glue together via affine diffeomorphisms. Moreover, the interior of each vertex is diffeomorphic either to a finite volume real-hyperbolic manifold with toral cusps (the pure pieces), or to the product of a standard torus with a lower dimensional finite volume real-hyperbolic manifold with toral cusps (the product or Seifert-like pieces). The goal of this paper is to show that the non-existence result for Einstein metrics on 3-manifolds with a non-trivial graph-like structure carries over to dimension four.

\begin{theorem}\label{maint}
	Closed extended graph $4$-manifolds do not support Einstein metrics.
\end{theorem}

Recall that in real dimension four an Einstein metric need not have constant sectional curvature, and we know a great deal concerning the existence, non-existence, and uniqueness of Einstein metrics. As in this paper we are concerned with obstructions to the existence of Einstein metrics, we simply refer to \cite{Bes87}, \cite{And06} and the very recent \cite{FP20} for some of the beautiful and highly non-trivial examples of Einstein metrics with non-constant sectional curvature. 

Regarding obstructions to the existence of such metrics, recall that the theory of Chern and Weil can be applied to derive the elegant Hitchin–Thorpe inequality \cite{Hit74}, which gives a necessary condition for the existence of an Einstein metric on a closed 4-manifold. LeBrun later pioneered the study of Einstein metrics on $4$-manifolds via the Seiberg–Witten equations, and among other things he was able to produce infinitely many closed $4$-manifolds (even simply connected) which satisfy the Hitchin-Thorpe inequality but nevertheless cannot support Einstein metrics \cite{LeB96}. This line of research has been extremely fruitful and even led to the computation of the Yamabe invariant of most complex surfaces \cite{LeB98}, and to some extent it was also generalized to the non-compact finite-volume setting \cite{Biq97}, \cite{Rol02}, \cite{DiC12}, \cite{DiC13}.

Another refinement of the Hitchin-Thorpe inequality was suggested earlier  by Gromov in \cite{Gro82} using the notion of simplicial volume. Sambusetti \cite{Sam98}, building up on results of Besson-Gallot-Curtois \cite{BCG95}, extended Gromov's obstruction and proved, among other things, that given any pair of integers $(k, t)$ such that $k-t\in 2\IZ$, then there exists an oriented $4$-manifolds $M$ with Euler characteristic $\chi(M)=k$ and signature $\tau(M)=t$ that cannot support an Einstein metric (see Theorem 4.4 in \cite{Sam98}). For a related circle of ideas we also refer to the work of Kotschick \cite{Kot98}. 

Thus, our main result adds to the vast array of known classes of manifolds which do not support Einstein metrics in dimension four. The proof employs some recent volume entropy estimates of Connell-Su\'arez-Serrato \cite{CSS19} for extended graph manifolds, and it is therefore a natural generalization of Sambusetti's work which uses the celebrated volume entropy estimate of Besson-Curtois-Gallot \cite{BCG95} for negatively curved locally symmetric spaces. With that said, we also crucially rely on a lemma of LeBrun (Lemma 8.1. in \cite{LeB99}) specialized to oriented Einstein $4$-manifolds with zero signature and negative cosmological constant. We refer to Section \ref{mainthm} for the details of the proof.

\begin{remark}
	Connell and Su\'arez-Serrato define a class of graph-like manifolds which they call \emph{higher graph manifolds}, see Definition 1 in \cite{CSS19}. The class of higher graph manifolds contains as a sub-class the extended graph manifolds defined by Frigerio, Lafont, and Sisto, see Remark 1 in \cite{CSS19}.  
\end{remark}

\begin{remark}
	For more information regarding the circle of ideas around of the volume entropy estimates in \cite{CSS19}, we refer to Souto's Ph.D. thesis \cite{Sou01}. Finally, for a very detailed treatment of the three-dimensional case the interested reader may also refer to Pieroni's Ph.D. thesis \cite{Pie18} (written under the supervision of Sambusetti), where the volume entropy of non-irreducible $3$-manifolds is also studied.
\end{remark}

Despite the considerable efforts, the study of Einstein metrics on manifolds of dimension $n\geq 5$ remains rather obscure when compared to dimension $n=4$. In fact, no uniqueness or non-existence results are currently known! As graph-like manifolds carry over their aversion to Einstein metrics from dimension three to dimension four, one may wonder whether or not this extends to dimension $n=5$ as well. More generally, we may ask the following.

\begin{question}
	Do extended graph $n$-manifolds with $n\geq 5$ support Einstein metrics?
\end{question}

We note that the volume entropy estimate of Connell-Su\'arez-Serrato \cite{CSS19} holds for extended graph manifolds not necessarily of dimension $n=4$, but unfortunately the other parts of the proof of Theorem \ref{maint} are special to this dimension.

\vspace{0.3 in}
\noindent\textbf{Acknowledgments}. The author would like to thank Professor Claude LeBrun for useful comments on an earlier version of this manuscript, and for generously sharing his knowledge of the subject over the years.\\ \\ \\

\section{Proof of the Main Theorem}\label{mainthm}

In this section, we give the details of the proof of Theorem \ref{maint}.
We follow the notation and curvature normalization adopted in LeBrun's paper \cite{LeB99}. We also refer to \cite{LeB99} as a reference for the necessary background on the interplay between the geometry and topology of $4$-manifolds and Einstein metrics.\\

Recall the Gauss-Bonnet formula for the Euler characteristic of a closed (oriented) $4$-manifold $(M, g)$ is given by
\[
\chi(M)=\frac{1}{8\pi^2}\int_{M}\Big(|W^+|_g^2 +|W^-|_g^2+\frac{s_g^2}{24}-\frac{|\overset{\circ}{Ric}|_g^2}{2}\Big)d\mu_g,
\]
where $W^{\pm}$ are the self-dual and anti-self-dual Weyl curvatures, $s_g$ is the scalar curvature, and $\overset{\circ}{Ric}$ is the trace-free part of the Ricci tensor. Moreover, because of the Hirzebruch signature theorem we can also express the signature of $(M, g)$ as a curvature integral as follows
\[
\tau(M)=\frac{1}{12\pi^2}\int_{M}\big(|W^+|_g^2-|W^-|_g^2\big)d\mu_g.
\]

Next, we define several \emph{minimal volumes} associated to a closed Riemannian manifolds. These were originally inspired by the definition of minimal volume in \cite{Gro82}, see Section 8 in \cite{LeB99} for more details. Given a closed Riemannian manifold $(M, g)$, we denote by $sec_g$ its sectional curvature, and by $Ric_g$ its Ricci tensor.

\begin{define}\label{minvolumes}
Let $M^n$ be a closed smooth manifold. Define the minimal volumes
\begin{itemize}
\item $Vol_{Sec}(M^n):=\inf_{g} \Big\{ Vol_{g}(M)\text{  } | \text{  } sec_g\geq -1\Big\};$
\item $Vol_{Ric}(M^n):=\inf_{g} \Big\{ Vol_{g}(M)\text{  } | \text{  } Ric_g\geq -(n-1)g\Big\};$
\item $Vol_{Scal}(M^n):=\inf_{g} \Big\{ Vol_{g}(M)\text{  } | \text{  } s_g\geq -n(n-1)\Big\}.$
\end{itemize}
Finally, we define the (Gromov) minimal volume of $M$ to be
\[
\text{MinVol(M)}:=\inf_{g} \Big\{ Vol_{g}(M)\text{  } | \text{  } |sec_{g}|\leq 1\Big\}.
\]
\end{define}

\begin{remark}\label{tautological}
	By definition, we have the chain of inequalities
	\[
	MinVol(M)\geq Vol_{Sec}(M)\geq Vol_{Ric}(M)\geq Vol_{Scal}(M)\geq 0.
	\]
\end{remark}

The next lemma is a useful strengthening of the original Hitchin-Thorpe inequality with the addition of a minimal Ricci volume term $Vol_{Ric}$. This lemma was observed in \cite{LeB99}, and the only addition here is an extra rigidity result under the assumption that the underlying $4$-manifold has zero signature and that the sign of the scalar curvature is negative.

\begin{lemma}[cf. Lemma 8.1. in \cite{LeB99}]\label{lebrun}
	Let $(M, g)$ be a $4$-dimensional closed Einstein manifold with negative scalar curvature. We then have
	\[
	2\chi(M)-3|\tau(M)|\geq\frac{3}{2\pi^2}Vol_{Ric}(M),
	\]
	with equality if and only if $g$ is half-conformally flat and it realizes the minimal Ricci volume up to rescaling. Moreover if $\tau(M)=0$ and the equality is achieved, then $M$ is a real-hyperbolic $4$-manifold.
\end{lemma}

\begin{proof}
	By using the Gauss-Bonnet and signature formulae, we know that
	\[
	2\chi(M)-3|\tau(M)|=\frac{1}{4\pi^2}\int_{M}\Big(2|W_{\mp}|_g^2+\frac{s_g^2}{24}\Big)d\mu_g,
	\]
	with $\mp$ depending on whether $\tau(M)$ is positive or negative. If $\tau(M)=0$, then there is no difference in selecting the self-dual or anti-self-dual Weyl curvature in the integrand. Next, we rescale the Einstein metric so that
	\[
	Rig_g=-3g \quad \Longrightarrow \quad s_g=-12.
	\]
	Thus, we have
	\[
	2\chi(M)-3|\tau(M)|\geq\frac{3}{2\pi^2}Vol_{g}(M)\geq \frac{3}{2\pi^2}Vol_{Ric}(M),
	\]
	with equality if and only if $g$ is half-conformally flat and it realizes the minimal Ricci volume. Finally, if $\tau(M)=0$ and if $g$ is Einstein and half-conformally flat, then $g$ is real-hyperbolic.
\end{proof}

We now study the Euler characteristic and signature of extended graph $4$-manifolds. We start by considering the case where there are no pure real-hyperbolic pieces, and then we study the case where pure pieces are allowed. 

\begin{lemma}\label{vanishing}
	Let $M$ be a closed extended graph $4$-manifold without pure pieces. We have $\chi(M)=\tau(M)=0$. If $M$ has real-hyperbolic pieces, we have $\tau(M)=0$ and $\chi(M)>0$.
\end{lemma}	
\begin{proof}
	If $M$ has no pure pieces, by Theorem 1 in \cite{CSS19} we know that 
	\[
	MinVol(M)=0.
	\]
	Let $\{g_{j}\}$ be a sequence of metrics on $M$ such that
	\[
	\lim_{j\to\infty}Vol_{g_{j}}(M)= 0,\quad |sec_{g_j}|\leq 1.
	\]
	By using the Hirzebruch signature formula we conclude
	\[
	\tau(M)=\lim_{j\to\infty}\frac{1}{12\pi^2}\int_{M}\big(|W^+|_{g_j}^2-|W^-|_{g_j}^2\big)d\mu_{g_j}=0,
	\]
	and by Gauss-Bonnet
	\[
	\chi(M)=\lim_{j\to\infty}\frac{1}{8\pi^2}\int_{M}\Big(|W^+|_{g_j}^2+|W^-|_{g_j}^2+\frac{s^2_{g_j}}{24}-\frac{|\overset{\circ}{Ric}|^2_{g_{j}}}{2}\Big)d\mu_{g_j}=0.
	\]
	This shows that $\chi(M)=\tau(M)=0$ when $M$ does not contain pure pieces.
	
	Next, let $M$ have $k\geq 1$ real-hyperbolic pieces say $\{(V_i, g_{-1})\}^k_{i=1}$. By Theorem 4 in \cite{CSS19}, we know that
	\[
	MinVol(M)=\sum^k_{i=1}Vol_{g_{-1}}(V_i)=\lim_{j\to\infty}Vol_{g_{j}}(M),
	\]
	where the sequence of metrics $\{g_{j}\}$ is collapsing with bounded curvature the non-pure pieces and the gluing regions, while being identically hyperbolic on larger and larger regions of the real-hyperbolic parts. For the explicit construction of the sequence of metrics $\{g_j\}$, we refer to the proofs of Theorem 3 and Theorem 4 in \cite{CSS19}. Now the real-hyperbolic metric has vanishing Weyl curvature, and as a result we conclude that
	\[
	\lim_{j\to\infty}\frac{1}{12\pi^2}\int_{M}\big(|W^+|_{g_j}^2-|W^-|_{g_j}^2\big)d\mu_{g_j}=0,
	\]
	and 
	\[
	\lim_{j\to\infty}\frac{1}{8\pi^2}\int_{M}\Big(|W^+|_{g_j}^2+|W^-|_{g_j}^2+\frac{s^2_{g_j}}{24}-\frac{|\overset{\circ}{Ric}|^2_{g_{j}}}{2}\Big)d\mu_{g_j}=\frac{3}{4\pi^2}\sum^{k}_{i=1}Vol_{g_{-1}}(V_i)>0.
	\]
	The proof of the lemma is complete.
\end{proof}

Next, we observe that we can explicitly compute the minimal Ricci volume of an extended graph manifold with pure real-hyperbolic pieces.

\begin{lemma}\label{extralemma}
	Let $M$ be an extended graph $4$-manifold with $k\geq 1$ pure real-hyperbolic pieces, say $\{(V_i, g_{-1})\}^k_{i=1}$. We then have
	\[
	Vol_{Ric}(M)=\sum^k_{i=1}Vol_{g_{-1}}(V_i).
	\]
\end{lemma}
\begin{proof}
	Consider an extended graph $4$-manifold $M$ with $k\geq 1$ pure real-hyperbolic pieces $\{(V_i, g_{-1})\}^k_{i=1}$, and let $g$ be a smooth Riemannian metric on $M$ normalized so that $Ric_g\geq -3g$. Given $(M, g)$, we define the volume entropy to be
	\begin{align}\label{entropy}
	h(g):=\lim_{R\to\infty}\frac{\log{Vol_{\tilde{g}}(B_{R}(p))}}{R},
	\end{align}
	where $B_{R}(p)$ is a ball in the Riemannian universal cover $(\widetilde{M}, \tilde{g})$. As shown by Manning in \cite{Man79}, the limit in Equation \eqref{entropy} exists and it is independent of the point $p\in\widetilde{M}$. Now any compact Riemannian $4$-manifold  with $Ric_g\geq -3g$ has volume entropy
	\[
	h(g)\leq 3,
	\]
	as follows from Bishop-Gromov volume comparison (see for example Chapter 7 in \cite{Pet16}). By Proposition 12 in \cite{CSS19}, given any such $(M, g)$ we have
	\begin{align}\label{connell}
	h(g)Vol_{g}(M)^{\frac{1}{4}}\geq 3(MinVol(M))^{\frac{1}{4}},
	\end{align}
	where because of Theorem 4 in \cite{CSS19} we have
	\begin{align}\notag
	MinVol(M)=\sum^k_{i=1}Vol_{g_{-1}}(V_i),
	\end{align}
	and where $\{(V_i, g_{-1})\}^k_{i=1}$ are the pure real-hyperbolic pieces of $M$, i.e., finite-volume real-hyperbolic $4$-manifolds with toral cusps equipped with the real-hyperbolic metric whose sectional curvature is normalized to be $-1$. Now the left-hand side in Equation \eqref{connell} is scale invariant, so that for any $g$ with $Ric_{g}\geq -3g$, we have that
	\begin{align}\notag
	Vol_{g}(M)\geq MinVol(M) \quad \Longrightarrow \quad Vol_{Ric}(M)=MinVol(M),
	\end{align}
	see Remark \ref{tautological}.
\end{proof}

We can now prove a non-existence result for Einstein metrics on extended graph $4$-manifolds (\emph{cf}. Theorem \ref{maint} in the Introduction). 

\begin{theorem}
	Let $M$ be a closed extended graph $4$-manifold. Then $M$ cannot admit an Einstein metric.
\end{theorem}

\begin{proof}
	 Consider first an extended graph $4$-manifold $M$ with $k\geq 1$ pure real-hyperbolic pieces say $\{(V_i, g_{-1})\}^k_{i=1}$. By Lemma \ref{extralemma}, we know that the minimal Ricci volume of $M$ is equal to sum of the real-hyperbolic volumes of the pure pieces whose sectional curvature is normalized to be $-1$
	 \begin{align}\label{mricci}
	 Vol_{Ric}(M)=\sum^k_{i=1}Vol_{g_{-1}}(V_i).
	 \end{align}
	 By using the finite-volume generalization of the Gauss-Bonnet formula (\emph{cf.} Harder \cite{Har71} and Cheeger-Gromov \cite{CG85}), we know that for any real-hyperbolic piece $(V_i, g_{-1})$
	\begin{align}\label{harder}
	\chi(V_i)=\frac{1}{8\pi^2}\int_{V_i}\frac{s_{g_{-1}}^2}{24}d\mu_{g_{-1}}=\frac{3}{4\pi^2}Vol_{g_{-1}}(V_i).
	\end{align}
    Next, a Mayer-Vietoris argument yields
    \begin{align}\label{vietoris}
    \chi(M)=\sum^k_{i=1}\chi(V_{i}),
    \end{align}
    so that by using Equations \eqref{mricci}, \eqref{harder}, and \eqref{vietoris} we have
    \begin{align}\notag
    2\chi(M)&=\frac{3}{2\pi^2}\sum^k_{i=1}Vol_{g_{-1}}(V_i)\\ \notag
    &=\frac{3}{2\pi^{2}}Vol_{Ric}(M).
    \end{align}
    Now by combining Lemma \ref{lebrun} and Lemma \ref{vanishing}, an Einstein metric $g$ on $M$ (rescaled so that $Ric_g=-3g$) realizes the minimal Ricci volume and it is then real-hyperbolic. The fundamental group of any extended graph $4$-manifold with more than one vertex and with at least one pure real-hyperbolic piece contains an abelian subgroup isomorphic to $\IZ^3$. To see this, recall that the cusps of a pure piece of real-hyperbolic type are toral and $\pi_1$-injective in $\pi_{1}(V_i)$. Now,  $\pi_1(M)$ can be represented as a graph of groups of its pieces, and as such $\pi_{1}(V_i)$ injects in $\pi_1(M)$. For the construction of $\pi_1(M)$ as a graph of groups we refer to Section 2.3. in \cite{LafontBook}. Thus, by Preissmann theorem (see for example Chapter 12 in \cite{doC92}) we have that $M$ cannot support a real-hyperbolic metric, and this concludes the proof in the case when $M$ has pure pieces of real-hyperbolic type.
    
    Finally if $M$ has no pure pieces, by Lemma \ref{vanishing} we know that
    \[
    \chi(M)=\tau(M)=0.
    \]
    Thus if $(M, g)$ is Einstein, it then saturates the Hitchin-Thorpe inequality
    \[
    2\chi(M)\geq 3|\tau(M)|,
    \]
    and as observed by Hitchin in \cite{Hit74} this forces the pull-back of $g$ to some finite cover of $M$ to be either a Ricci-flat K\"ahler metric on a K3 surface or a flat metric on a $4$-torus. By a well-known observation of Milnor (see Theorem 1 in \cite{Mil68}), the growth of $\pi_1(M)$ is then polynomial. On the other hand, the fundamental group of an extended graph manifold grows exponentially, see Proposition 6.16 in \cite{LafontBook}. The proof is complete.
\end{proof}

\vspace{0.1in}

\end{document}